\documentclass[12pt, english]{amsart}

\usepackage{dsfont,amssymb,amsfonts}

\makeatletter
\def\msection{\@startsection{section} 
	{1} 
	{0pt} 
	{-1ex plus -.1ex minus -0.9ex} 
	{-.9ex plus -.2ex} 
	{\bfseries} 
}
\def\msubsection{\@startsection{subsection} 
	{2} 
	{0pt} 
	{-1ex plus -.1ex minus -0.2ex} 
	{-.9ex plus -.2ex} 
	{\normalfont} 
} 
\makeatother 

\usepackage[text={6in,9in},centering]{geometry}

\usepackage[dvipsnames]{xcolor}   
\usepackage{xparse}
\usepackage{xr-hyper}
\definecolor{brightmaroon}{rgb}{0.76, 0.13, 0.28}
\usepackage[linktocpage=true,colorlinks=true,hyperindex,citecolor=BrickRed,linkcolor=blue]{hyperref}

\newtheorem{theorem}{Theorem}[section] 

\newtheorem{lemma}[theorem]{Lemma}
 
\theoremstyle{definition}

\newtheorem{remark}[theorem]{Remark}

\begin{document}

\title{Spectral spaces of normal subgroups}





\author{Amartya Goswami}
\address{[1] Department of Mathematics and Applied Mathematics, University of Johannesburg, South Africa.
[2] National Institute for Theoretical and Computational Sciences (NITheCS), Johannesburg, South Africa.}
\email{agoswami@uj.ac.za}
\maketitle

\begin{abstract} The aim of this note is to prove that the set of proper normal subgroups of a group endowed with coarse lower topology is a spectral space. 
\end{abstract}

\smallskip

\section{Introduction and Preliminaries} 
\smallskip

According to \cite{H69}, a \emph{spectral space} is a topological space that is  quasi-compact, sober,  admitting a basis of quasi-compact open subspaces that is closed under finite intersections.  The following interesting result in \cite[Proposition 3.9]{FGT22} completely characterizes the spectral space of $\mathrm{Spec}(G)$ of prime normal subgroups of a group $G$. The spectrum $\mathrm{Spec}(G)$ endowed with Zariski topology is spectral if and only if it is compact, if and
only if $G$ has a maximal normal subgroup. Our aim in this paper is to prove the following.

\begin{theorem}\label{mth}
Let $G$ be a group having a maximal normal subgroup. Then the set $\mathcal{N}^+(G)$ of proper normal subgroups of $G$ endowed with coarse lower topology is a spectral space.
\end{theorem}

In this paper, the identity element of a group $G$ will be denoted by $1$. The notation $\mathcal{N}(G)$ will denote the set of all normal subgroups of $G$.  If $\{N_{\lambda}\}_{\lambda \in \Lambda}$ is a collection of normal subgroups of $G$, then we shall denote their \emph{join} (\textit{i.e.}, the normal subgroup generated by $\bigcup_{\lambda \in \Lambda}N_{\lambda}$) by   $\bigvee_{\lambda \in \Lambda}N_{\lambda},$ which is also a normal subgroup of $G$. We say a proper normal subgroup $M$ of a group $G$ is \emph{maximal} if there is no proper normal subgroup of $G$ that properly contains $M$. 
Note that any nontrivial finite group has maximal normal
subgroups. 
However, this may not be
true for infinite groups, as the example of the Prüfer group shows.
This is in contrast with unitary rings, where the axiom of choice always guarantees existence of maximal ideals. 

Let us now recall a few definitions from topology. A space is called \emph{quasi-compact} if every open cover of it has a finite subcover, or equivalently, the space satisfies the finite intersection property. In this definition of quasi-compactness, we do not assume the space is $T_2.$ A closed subset $S$ of a space $X$ is called \emph{irreducible} if $S$ is not
the union of two properly smaller closed subsets of X and $S\neq \emptyset$.  A  space $X$ is called \emph{sober} if every irreducible closed subset $\mathcal K$ of $X$ is of the form: $\mathcal K=\mathcal{C}(\{x\})$, the closure of a unique singleton set $\{x\}$.

\section{The space $\mathcal{N}^+(G)$}
\smallskip

Let $G$ be a group such that $G$ is the normal closure of a finitely generated subgroup of $G$. The coarse lower topology on  $\mathcal{N}(G)$ is the topology for which the sets of the type 
\[
\mathcal{V}(S)=
\{N\in \mathcal{N}(G)\mid S\subseteq N \}, \qquad  (S\in \mathcal{N}(G))	
\]
form a subbasis of closed sets, 
and we use the  same notation $\mathcal{N}(G)$ to denote the space. The set $\mathcal{N}^+(G)$ with the subspace topology of the space $\mathcal{N}(G)$ is our object of study. To prove Theorem \ref{mth} for $\mathcal{N}^+(G)$, we need to show the following:
\begin{enumerate}
\item The space $\mathcal{N}^+(G)$ is quasi-compact.

\item The space $\mathcal{N}^+(G)$ is sober.

\item The coarse lower topology on $\mathcal{N}^+(G)$ admits a basis of quasi-compact open subspaces that is closed under finite intersections.
\end{enumerate}

Thanks to the next lemma, the checking of (3) can be avoided.
\smallskip

\begin{lemma}\label{cso}
A quasi-compact, sober, open subspace of a spectral space is spectral. 
\end{lemma}

\begin{proof}
Suppose that $S$ is a quasi-compact, sober, open subspace of a spectral space $X$. Since $S$ is quasi-compact and sober, it is sufficient to prove that the set $\mathcal{O}_{\scriptscriptstyle S}$ of compact open subsets of $S$ forms a basis of a topology that is closed under finite intersections. It is obvious that a subset $T$ of $S$ is open in $S$ if and only if $T$ is open in $X$, and hence a subset $T$ of $S$ belongs to $\mathcal{O}_{\scriptscriptstyle S}$ if and only if $T$ belongs to $\mathcal{O}_{\scriptscriptstyle X}.$ Now 
using these facts, we argue as follows.
Let $U$ be an open subset of $S$. Since $U$ is also open in $X$, we have $U=\cup\, \mathcal{U}$ for some subset $\mathcal{U}$ of $\mathcal{O}_{\scriptscriptstyle X}.$ But each element of $\mathcal{U}$ being a subset of $U$ is a subset of $S$, and it belongs to $\mathcal{O}_{\scriptscriptstyle S}.$ Therefore, every open subset of $S$ can be presented as a union of compact open subsets of $S$. Now it remains to prove that $\mathcal{O}_{\scriptscriptstyle S}$ is closed under finite intersections, but this immediately follows from the fact that $\mathcal{O}_{\scriptscriptstyle X}$ is closed under finite intersections. 
\end{proof} 
\smallskip

It is now clear that to prove Theorem \ref{mth}, it is sufficient to check the conditions in Lemma \ref{cso}. Therefore, our objective is now to verify the following:

\begin{enumerate}

\item \label{ngs}
$\mathcal{N}(G)$ is a spectral space;

\item \label{oco} $\mathcal{N}^+(G)$ is quasi-compact;

\item \label{tso} $\mathcal{N}^+(G)$ is sober.

\item \label{top}  $\mathcal{N}^+(G)$ is an open subspace of the space $\mathcal{N}(G)$.
\end{enumerate}

\begin{proof} 
1. Since $\mathcal{N}(G)$ is an algebraic lattice, the desired spectrality follows from \cite[Theorem 4.2]{P94}.

2.
Let  $\{K_{ \lambda}\}_{\lambda \in \Lambda}$ be a family of subbasic closed sets of a normal structure space $\mathcal{N}^+(G)$   such that $\bigcap_{\lambda\in \Lambda}K_{ \lambda}=\emptyset.$ Let $\{N_{ \lambda}\}_{\lambda \in \Lambda}$ be a family of normal subgroups of $\mathcal{N}(G)$ such  that $\forall \lambda \in \Lambda,$  $K_{ \lambda}=\mathcal{V}(N_{ \lambda}).$  Since $\bigcap_{\lambda \in \Lambda}\mathcal{V}(N_{ \lambda})=\mathcal{V}\left(\bigcup_{\lambda \in \Lambda}N_{ \lambda}\right),$ we get  $\mathcal{V}\left(\bigcup_{\lambda \in \Lambda}N_{ \lambda}\right)=\emptyset.$ This implies that the normal subgroup $ \bigvee_{\lambda \in \Lambda}N_{ \lambda}$ generated by $\bigcup_{\lambda \in \Lambda}N_{ \lambda}$ must be equal to $G$. Then,  by hypothesis on $G$ and using Alexander Subbasis Theorem, we have the desired quasi-compactness.  	
	
3. To show the existance of generric points of irreducible closed subsets of $\mathcal{N}^+(G)$, it is sufficient to show that $\mathcal{V}(N)=\mathcal{C}(N)$, whenever $N\in\mathcal{V}(N)$. Since $\mathcal{C}(N)$ is the smallest closed set containing $N$, and since $\mathcal{V}(N)$ is a closed set containing $N$, obviously then  $\mathcal{C}(N)\subseteq \mathcal{V}(N)$. 
For the reverse inclusion, if $\mathcal{C}(N)= \mathcal{N}^+(G)$, then 
\[ 
\mathcal{N}^+(G)=\mathcal{C}(N)\subseteq \mathcal{V}(N)\subseteq \mathcal{N}^+(G).
\] 
This proves that $\mathcal{V}(N)=\mathcal{C}(N)$. Suppose that $\mathcal{C}(N)\neq \mathcal{N}^+(G)$. Since $\mathcal{C}(N)$ is a closed set,  there exists an  index set, $\Lambda$, such that,  for each $\lambda\in\Lambda$, there is a positive integer $n_{\lambda}$ and normal subgroups $N_{\lambda 1},\dots, N_{\lambda n_\lambda}$ of $G$ such that 
\[
\mathcal{C}(N)={\bigcap_{\lambda\in\Lambda}}\left({\bigcup_{ i\,=1}^{ n_\lambda}}\mathcal{V}(N_{\lambda i})\right).
\]
Since  
$\mathcal{C}(N)\neq \mathcal{N}^+(G),$ we  assume that ${\bigcup_{ i\,=1}^{ n_\lambda}}\mathcal{V}(N_{\lambda i})$ is non-empty for each $\lambda$. Therefore, $N\in   {\bigcup_{ i\,=1}^{ n_\lambda}}\mathcal{V}(N_{\lambda i})$ for each $\lambda$, and hence $\mathcal{V}(N)\subseteq {\bigcup_{ i=1}^{ n_\lambda}}\mathcal{V}(N_{\lambda i})$, that is $\mathcal{V}(N)\subseteq \mathcal{C}(N)$ as desired. 

To obtain the uniqueness of the generic point, it is sufficant to prove that $\mathcal{N}^+(G)$ is a $T_0$-space. Let $N$ and $N'$ be two distinct elements of $\mathcal{N}^+(G)$. Then, without loss of generality, we may assume that $N\nsubseteq N'$. Therefore $\mathcal{V}(N)$ is a closed set containing $N$ and missing $N'$. 

4. Since $G\in\mathcal{N}(G),$  it follows that $G=\mathcal{V}(G)=\mathcal{C}(G),$ and therefore $\mathcal{N}(G) {\setminus}\mathcal{N}^+(G)$ is closed, and that implies $\mathcal{N}^+(G)$ is open. 
\end{proof}  

\begin{remark}
Among some distinguished classes of ideals of a ring endowed with coarse lower topology, in \cite{FGS22}, those are spectral have been studied. Similarly, it would be interesting to investigate and identify various classes of spectral spaces associated to normal subgroups of a group.
\end{remark}


\begin{thebibliography}{9} 

\bibitem{FGS22} C. A. Finocchiaro, A. Goswami, and D. Spirito, ``Distinguished classes of ideal spaces and their topological properties'', \textit{Comm. Algebra}, 51(4), (2023), 1752--1760. 

\bibitem{FGT22} A. Facchini, F. de Giovanni, and M. Trombetti, ``Spectra of groups'', \textit{Algebr. Represent. Theory}, 26(5), (2023), 1415--1431.

\bibitem{H69} 
M. Hochster, ``Prime ideal structure in commutative rings'', \textit{Trans. Am. Math. Soc.}, 142 (1969), 43--60.

\bibitem{P94} H. A. Priestley, ``Intrinsic spectral topologies'', in: \textit{Papers on general topology and applications}
(Flushing, NY, 1992), 728,  78--95, New York Acad. Sci., New York, 1994.

\end{thebibliography}
\end{document}